\documentclass[11pt]{article}
\usepackage{mathrsfs}
\usepackage{amssymb,amsfonts,amsmath, psfrag,eepic,colordvi,epsfig}
\usepackage{CJK}
\usepackage{booktabs}
\topskip 
\numberwithin{equation}{section}
\newtheorem{theorem}{Theorem}[section]
\newtheorem{proposition}[theorem]{Proposition}

\newtheorem{lemma}[theorem]{Lemma}

\newenvironment{proof}
   {\par\noindent{\it Proof.}\hskip 0.5em\ignorespaces}
   {\hfill \par\medskip}

\begin{document}
\begin{CJK*}{UTF8}{gkai}
\parskip 7pt

\pagenumbering{arabic}
\def\sof{\hfill\rule{2mm}{2mm}}
\def\ls{\leq}
\def\gs{\geq}
\def\SS{\mathcal S}
\def\qq{{\bold q}}
\def\MM{\mathcal M}
\def\FF{\mathcal F}
\def\EE{\mathcal E}
\def\lsp{\mbox{lsp}}
\def\rsp{\mbox{rsp}}
\def\pf{\noindent {\it Proof.} }
\def\mp{\mbox{pyramid}}
\def\mb{\mbox{block}}
\def\mc{\mbox{cross}}
\def\qed{\hfill \rule{4pt}{7pt}}
\def\pf{\noindent {\it Proof.} }
\textheight=22cm

\begin{center}
{\Large\bf A proof of the only mode of a unimodal sequence}
\end{center}
\begin{center}
Jun Wan, Zuo-Ru Zhang \\

School of Mathematical Science\\
Hebei Normal University\\
 Shijiazhuang 050024, P. R. China\\[6pt]
\end{center}

\noindent {\bf Abstract.}
In study the generalized Jacobsthal and Jaco-Lucas polynomials, Sun introduced the interesting numerical triangle Jaco-Lucas sequence $\{JL_{n,k}\}_{n\geq k\geq0}$. In this paper, we proved this sequence is log-concave with the only  mode by computer algebra.


\section{Introduction}

The study of unimodality of combinatorial sequences has drawn great attention in recent decades see the survey of Stanley \cite{PS1989} and Brenti \cite{FB1994,FB1989}.
Specially  sequences with  finite order recurrences \cite{hou2019,hou2021,Mao2023,Zhu2020} gets a lot of attentions.
Proving unimodality of sequence is always difficult, especially given the mode of the sequence as was wrote by Stanely in \cite{PS1989}.
In study a class of polynomial sequences with a three term recurrence, Sun\cite{Sun1975} introduced several numerical triangles. Among these sequences we find the  Jaco-Lucas sequence $\{JL_{n,k}\}_{n\geq k\geq0}$ satisfy a three term recurrence with  quadratic coefficients.

It is well known that  for a sequence $\mathcal{A}=\{a_k\}_{k=0}^n$  of positive numbers, we have

\noindent$\bullet$ $\mathcal{A}$ is unimodal if there is an index $0\leq j\leq n$ such that
\begin{equation*}
    a_0\leq a_1\leq \dots \leq a_j\geq a_{j+1}\geq \dots \geq a_n.
\end{equation*}

\noindent$\bullet$ $\mathcal{A}$ is log-concave if
\begin{equation*}
    a_j^2\geq a_{j-1}a_{j+1},~~~~for~all~1\leq j< n.
\end{equation*}

The Jaco-Lucas sequence $\{JL_{n,k}\}_{n\geq k\geq0}$ defined as
\begin{equation}\label{1.1}
    JL_{n,k}=\sum_{i=k}^{\lfloor\frac{n}{2}\rfloor}\frac{n}{n-i}\binom{n-i}{i}\binom{i}{k}.
\end{equation}
also has a combinatorial interpretation given by Sun\cite{Sun1975}  as  the number of $(0,1,2)$-sequences of length $n$ with $k$ $2$'$s$, but without subsequences $11$, $12$, $21$, $22$, where the first and last components of the sequences are considered to be adjacent .

Inspired by Kauers and Paule\cite{KP2007} we proved that the Jaco-Lucas sequeces $\{JL_{n,k}\}_{n\geq k\geq0}$ also can be find as A245962 \cite{NJ} in OEIS is log-concavity,  at the same time we give its only mode by cylindrical algebraic decomposition (CAD) \cite{EC1975}. Specially we prove that  the sequence is strictly increasing in the left of the mode.

By observation of the first few values of this sequence we conjecture that  the sequence $\{JL_{n,k}\}_{k\geq0}$ is unimodal with mode $k^*=\lfloor \frac{n-4}{6} \rfloor+1$. The values of $JL_{n,k}$ for small $n$ and $k$ can be found in Table 1.
\begin{table}[t]
    \centering
    \caption{The values of $JL_{n,k}$ for $1\leq n\leq 18$ and $0\leq k\leq \lfloor n/2 \rfloor$.}
    \label{tab:my_label}
    \begin{tabular}{*{11}{c}}
    \toprule
     &\multicolumn{10}{c}{k}\\
      \cmidrule(lr){2-11}
    n&0&1&2&3&4&5&6&7&8&9\\ \midrule
     1	&1\\
     2	&3	&2\\
     3	&4	&3\\
     4	&7	&8	&2\\
     5	&11	&15	&5\\
     6	&18	&30	&15	&2\\
     7	&29	&56	&35	&7\\
     8	&47	&104	&80	&24	&2\\
     9	&76	&189	&171	&66	&9\\
     10	&123	&340	&355	&170	&35	&2\\
     11	&199	&605	&715	&407	&110	&11\\
     12	&322	&1068	&1410	&932	&315	&48	&2\\
     13	&521	&1872	&2730	&2054	&832	&169	&13\\
     14	&843	&3262	&5208	&4396	&2079	&532	&63	&2\\
     15	&1364	&5655	&9810	&9180	&4965	&1533	&245	&15\\
     16	&2207	&9760	&18280	&18784	&11440	&4144	&840	&80	&2\\
     17	&3571	&16779	&33745	&37774	&25585	&10642	&2618	&340	&17\\
     18	&5778	&28746	&61785	&74838	&55809	&26226	&7602	&1260	&99	&2\\
    \bottomrule
    \end{tabular}
\end{table}
We confirm the conjecture and give the main theorem below.
\begin{theorem}\label{thm1.2} For any positive integers $n$ and $k$ with $n\geq 2k> 0$, $JL_{n,k}$ is log-concavity and strictly increasing before the only mode $k^*=\lfloor \frac{n-4}{6} \rfloor+1$.
\end{theorem}
\section{The proof of log-concavity}

In this section we prove that the sequence $\{JL_{n,k}\}_{k\geq 0}$ is log-concavity with the help of packages from RISC. Based on the recurrence relations for $JL_{n,k}$,  we give the bound for the ratio $\frac{JL_{n,k+1}}{JL_{n,k}}$.

\subsection{Recurrence relations of $JL_{n,k}$}

\begin{lemma}\label{lem2.1}
For any positive integers $n$ and $k$ with $n\geq 2k$, we have
\begin{equation}\label{2.1}
    \begin{split}
    &JL_{n,k-1}=\frac{kn}{(n+1)(n-2k+2)}JL_{n+1,k}+\frac{2k}{n-2k+2}JL_{n,k},
    \end{split}
\end{equation}
\begin{equation}\label{2.2}
    JL_{n,k+1}=-\frac{n(n-2k+1)}{5(n+1)(k+1)}JL_{n+1,k}+\frac{3n-5k}{5(k+1)}JL_{n,k},
\end{equation}
\begin{equation}\label{2.3}
    JL_{n+2,k}=\frac{(n-k+1)(n+2)}{(n+1)(n-2k+2)}JL_{n+1,k}+\frac{n+2}{n-2k+2}JL_{n,k}.
\end{equation}
\end{lemma}

\begin{proof}
We first prove \eqref{2.1} by using the RISC package MultiSum a suitable recurrence for the $b_{n,k}$. MultiSum is an implementation of Wegschaider's algorithm, which is an extension of multivariate WZ summation \cite{SW1992}. After loading the package into Mathematica by

\noindent$\scriptsize{\textsf{In[1]}:=}$ \textbf{$<<$ RISC `MultiSum.m'}

~~~~~~~~~~~Package MultiSum version 2.3 written by Kurt Wegschaider

~~~~~~~~~~~enhanced by Axel Riese and Burkhard Zimmermann

~~~~~~~~~~~Copyright Research Institute for Symbolic Computation (RISC),

~~~~~~~~~~~Johannes Kepler University, Linz, Austria

we input the summand of the sum \eqref{1.1}

\noindent$\scriptsize{\textsf{In[2]:=}}$ $\boldsymbol{g:=\frac{n}{n-i}}$\textbf{Binomial}$\boldsymbol{[n-i,i]}$\textbf{Binomial}$\boldsymbol{[i,k]}$

\noindent$\scriptsize{\textsf{In[3]:=}}$ \textbf{regions=FindStructureSet[}$\boldsymbol{g, \{n, k\}, \{1,0\}, \{i\}, \{1\}, 1]}$

\noindent$\scriptsize{\textsf{Out[3]=}}$ Omit

\noindent$\scriptsize{\textsf{In[4]:=}}$ \textbf{FindRecurrence[}$\boldsymbol{g, \{n, k\}, \{i\},}$ \textbf{regions}$\boldsymbol{[[1]], 1,}$ \textbf{WZ$\to$True]}

\noindent$\scriptsize{\textsf{Out[4]=}}$ \{$-n(1+n)F[-1+n, -1+k, -1+i]-k(1+n)F[n, k, -1+i]+2knF[1+n, k, -1+i]==\Delta_i[k(1+n)F[n, k, -1+i]-2knF[1+n, k, -1+i]]$\}

\noindent$\scriptsize{\textsf{In[5]:=}}$ \textbf{SumCertificate[\%]}

\noindent$\scriptsize{\textsf{Out[5]=}}$ \{$-n(1+n)$SUM$[-1+n, -1+k]-k(1+n)$SUM$[n, k]+2kn$SUM$[1+n, k]==0$\}

In other words, we have that for $n\geq 2k\geq 0$,
\begin{equation}\label{2.4}
    -n(1+n)JL_{n-1,k-1}-k(1+n)JL_{n,k}+2knJL_{n+1,k}=0.
\end{equation}
By replacing $n$ with $n+1$ in \eqref{2.4}, we get
\begin{equation}\label{2.5}
    JL_{n,k-1}=-\frac{k}{n+1}JL_{n+1,k}+\frac{2k}{n+2}JL_{n+2,k}.
\end{equation}

By using the RISC package $\mathbf{fastZeil.m}$ which is developed by Paule, Schorn, and Riese \cite{PM1992,AR2001},

\noindent$\scriptsize{\textsf{In[6]:=}}$ \textbf{$<<$ RISC `fastZeil'}

~~~~~~~~~~Fast Zeilberger Package version 3.61

~~~~~~~~~~written by Peter Paule, Markus Schorn, and Axel Riese

~~~~~~~~~~Copyright Research Institute for Symbolic Computation (RISC),

~~~~~~~~~~Johannes Kepler University, Linz, Austria

\noindent$\scriptsize{\textsf{In[7]:=}}$ \textbf{Zb[}$\boldsymbol{g, \{i, k,}$ \textbf{Infinity\},} $\boldsymbol{k, 2]}$

\noindent$\scriptsize{\textsf{Out[7]=}}$ $\{-((2k-n)(1+2k-n)$SUM$[k])-(1+k)(7+9k-5n)$SUM$[1+k]-$$5(1+k)(2+k)$SUM$[2+k]==0\}$

In other words, we have the following formula for $n\geq 2k\geq 0$,
\begin{equation}\label{2.6}
    JL_{n,k-1}=\frac{k(5n-9k+2)}{(n-2k+2)(n-2k+1)}JL_{n,k}-\frac{5k(k+1)}{(n-2k+2)(n-2k+1)}JL_{n,k+1}.
\end{equation}
\noindent$\scriptsize{\textsf{In[8]:=}}$ \textbf{Zb}$\boldsymbol{[g, \{i, k,}$ \textbf{Infinity\}}$\boldsymbol{, n, 2]}$

\noindent$\scriptsize{\textsf{Out[8]=}}$ $\{-((1+n)(2+n)$SUM$[n])+(-1+k-n)(2+n)$SUM$[1+n]+$$(1+n)(2-2k+n)$SUM$[2+n]==0\}$

In other words, we have that for $n\geq 2k\geq 0$,
\begin{equation*}
    -(1+n)(2+n)JL_{n,k}+(-1+k-n)(2+n)JL_{n+1,k}+(1+n)(2-2k+n)JL_{n+2,k}=0.
\end{equation*}
which is \eqref{2.3}.

Therefore, replace $JL_{n+2,k}$ with \eqref{2.3} in \eqref{2.5} we can obtain the recurrence relation \eqref{2.1}. Finally, by substituting \eqref{2.1} for $JL_{n,k-1}$ in \eqref{2.6}, we obtain the recurrence relation \eqref{2.2}. This completes the proof. $\Box$
\end{proof}

\subsection{Bounds of $\frac{JL_{n+1,k}}{JL_{n,k}}$}

\begin{lemma}\label{lem2.2} For any positive integers $n$ and $k$ with $n\geq 2k$, we have
\begin{equation*}
    L(n,k)\leq \frac{JL_{n+1,k}}{JL_{n,k}},
\end{equation*}
where
\begin{equation*}
    L(n,k)=\frac{(n-k)(n+1)}{n(n-2k+1)}.
\end{equation*}
\end{lemma}

\begin{proof}
Fixing $k$, we prove this lemma by using inductions on $n$. For the case $n=2k$, we need to show
\begin{equation*}
    L(2k,k)\leq \frac{JL_{2k+1,k}}{JL_{2k,k}}.
\end{equation*}
By the definition,
\begin{equation*}
    \begin{split}
    &\frac{JL_{2k+1,k}}{JL_{2k,k}}=\frac{2k+1}{2},\\&
    L(2k,k)=\frac{2k+1}{2}.
    \end{split}
\end{equation*}
Therefore, the desired inequality holds for $n=2k$.

Assume that the inequality holds for some $n\geq 2k$, namely,
\begin{equation*}
     L(n,k)\leq \frac{JL_{n+1,k}}{JL_{n,k}}.
\end{equation*}
We proceed to prove the desired inequalities hold for $n+1$ as well. By dividing both sides of \eqref{2.3} by $JL_{n+1,k}$, we have
\begin{equation}\label{2.8}
    \frac{JL_{n+2,k}}{JL_{n+1,k}}=\frac{2+n}{n-2k+2}\frac{JL_{n,k}}{JL_{n+1,k}}+\frac{(n-k+1)(n+2)}{(1+n)(n-2k+2)}.
\end{equation}
Since $n-2k+2>0$ for $n\geq 2k$ we obtain
\begin{equation}\label{2.9}
    \frac{2+n}{n-2k+2}\frac{JL_{n,k}}{JL_{n+1,k}}>0.
\end{equation}
And we know
\begin{equation}\label{2.10}
    L(n+1,k)=\frac{(n-k+1)(n+2)}{(1+n)(n-2k+2)}.
\end{equation}
By \eqref{2.8}, \eqref{2.9} and \eqref{2.10} we obtain
\begin{equation*}
     \frac{JL_{n+2,k}}{JL_{n+1,k}}\geq L(n+1,k).
\end{equation*}
This completes the proof. $\Box$
\end{proof}

\begin{lemma}\label{lem2.3}
For any positive integers $n$ and $k$ with $n\geq 19$, $n\geq 2k$, $k_1<k<k_2$, we have
\begin{equation}
    l(n,k)\leq \frac{JL_{n+1,k}}{JL_{n,k}}\leq h(n,k),
\end{equation}
where
\begin{equation*}
    \begin{split}
    &k_1=\frac{20+9n+n^2}{4(-1+n)}-\frac{1}{4}\sqrt{\frac{480+400n+41n^2-22n^3+n^4}{(-1+n)^2}},\\&
    k_2=\frac{20+9n+n^2}{4(-1+n)}+\frac{1}{4}\sqrt{\frac{480+400n+41n^2-22n^3+n^4}{(-1+n)^2}},\\&
    l(n,k)= -\frac{2+k+n+kn-n^2}{2n(1-2k+n)}+\frac{1}{2}\sqrt{\frac{H_1(n,k)}{k(-1+2k-n)^2n^2}},\\&
    h(n,k)=\left(n-k+\frac{2\sqrt{k}(2k-n)(-1+n)}{k^{3/2}-\sqrt{H_2(n,k)}-\sqrt{k}(-3+n)}
    \right)\cdot \frac{1+n}{n(1-2k+n)},\\&
    H_1(n,k)=-40+84k+4k^2+k^3-140n+208kn-2k^2n+2k^3n-180n^2+169kn^2-16k^2n^2\\&+k^3n^2-100n^3+50kn^3-10k^2n^3-20n^4+5kn^4,\\&
    H_2(n,k)=49k+14k^2+k^3-20n+30kn-10k^2n-20n^2+5kn^2.\\&
    \end{split}
\end{equation*}
\end{lemma}

\begin{proof}
Fixing $k$, we prove this lemma by using induction on $n$. For the case $n=2k$, we need to show
\begin{equation*}
    l(2k,k)\leq \frac{JL_{2k+1,k}}{JL_{2k,k}}\leq h(2k,k).
\end{equation*}
By the definition,
\begin{equation*}
    \begin{split}
    &\frac{JL_{2k+1,k}}{JL_{2k,k}}=\frac{2k+1}{2},\\&
    l(2k,k)=-\frac{2+3k-2k^2}{4k}+\frac{1}{4}\sqrt{\frac{-40-196k-300k^2-127k^3+20k^4+4k^5}{k^3}},\\&
    h(2k,k)=\frac{2k+1}{2}.
    \end{split}
\end{equation*}
We need to confirm that $\frac{JL_{2k+1,k}}{JL_{2k,k}}\geq l(2k,k)$, then

\noindent$\scriptsize{\textsf{In[9]:=}}$ $\boldsymbol{k_1:=\frac{20+9n+n^2}{4(-1+n)}-\frac{1}{4}\sqrt{\frac{480+400n+41n^2-22n^3+n^4}{(-1+n)^2}}}$

\noindent$\scriptsize{\textsf{In[10]:=}}$ $\boldsymbol{k_2:=\frac{20+9n+n^2}{4(-1+n)}+\frac{1}{4}\sqrt{\frac{480+400n+41n^2-22n^3+n^4}{(-1+n)^2}}}$

\noindent$\scriptsize{\textsf{In[11]:=}}$ $\boldsymbol{l[2k\_,k\_]:=-\frac{2+3k-2k^2}{4k}+\frac{1}{4}\sqrt{\frac{-40-196k-300k^2-127k^3+20k^4+4k^5}{k^3}}}$

\noindent$\scriptsize{\textsf{In[12]:=}}$ \textbf{CylindricalDecomposition[Implies[}$\boldsymbol{n\geq 19\&\&k_1<k<k_2,\frac{1}{2}(1+2k)-l[2k,k]\geq 0], \{n, k\}]}$

\noindent$\scriptsize{\textsf{Out[12]=}}$ True

Therefore, the desired inequality holds for $n=2k$.

Assume that the inequality holds for some $n\geq 2k$, namely,
\begin{equation*}
    l(n,k)\leq \frac{JL_{n+1,k}}{JL_{n,k}}\leq h(n,k).
\end{equation*}
We proceed to prove the desired inequalities hold for $n+1$ as well. By \eqref{2.8}, since $n-2k+2>0$ for $n\geq 2k$ we obtain
\begin{equation}\label{2.12}
    \begin{split}
    &l(n+1,k)\leq \frac{2+n}{n-2k+2}\frac{1}{h(n,k)}+\frac{(n-k+1)(n+2)}{(1+n)(n-2k+2)}\leq \frac{JL_{n+2,k}}{JL_{n+1,k}}\\&
    \leq \frac{2+n}{n-2k+2}\frac{1}{l(n,k)}+\frac{(n-k+1)(n+2)}{(1+n)(n-2k+2)}\leq h(n+1,k).
    \end{split}
\end{equation}
By the definition of $h(n,k)$, we can see the right side of \eqref{2.12} is equal to $h(n+1,k)$. We only need to show the left side of \eqref{2.12} is greater than or equal to $l(n+1,k)$, namely,
\begin{equation*}
    l(n+1,k)\leq \frac{2+n}{n-2k+2}\frac{1}{h(n,k)}+\frac{(n-k+1)(n+2)}{(1+n)(n-2k+2)}.
\end{equation*}
The above formula is equivalent to proof
\begin{equation}\label{2.13}
    \begin{split}
    &\frac{2+n}{n-2k+2}\frac{1}{h(n,k)}+\frac{(n-k+1)(n+2)}{(1+n)(n-2k+2)}-l(n+1,k)\geq 0.
    \end{split}
\end{equation}
Since $n\geq 2k, n\geq 19,k_1<k<k_2$, \eqref{2.13} is equivalent to
\begin{equation}\label{2.14}
    \begin{split}
    &G(n,k)=\frac{1+k-n}{4-4k+2n}+\frac{1-k+n}{2-2k+n}-\\&
    \frac{\sqrt{k^3-20(2+n)(3+n)-2k^2(3+5n)+k(129+5n(10+n))}}{\sqrt{k}(4-4k+2n)}+\\&
    \frac{\left(\frac{n(1-2k+n)}{\left(-k+n+\frac{2\sqrt{k}(2k-n)(-1+n)}{\sqrt{k}(3+k-n)-\sqrt{k(7+k)^2-10(-2+k)(-1+k)n+5(-4+k)n^2}}\right)}\right)}{2-2k+n}\geq 0.
    \end{split}
\end{equation}
By using CAD,

\noindent$\scriptsize{\textsf{In[13]:=}}$ $\boldsymbol{G[n\_,k\_]:=\frac{1+k-n}{4-4k+2n}+\frac{1-k+n}{2-2k+n}-\frac{\sqrt{k^3-20(2+n)(3+n)-2k^2(3+5n)+k(129+5n(10+n))}}{\sqrt{k}(4-4k+2n)}+}$
$\boldsymbol{\frac{\left(\frac{n(1-2k+n)}{\left(-k+n+\frac{2\sqrt{k}(2k-n)(-1+n)}{\sqrt{k}(3+k-n)-\sqrt{k(7+k)^2-10(-2+k)(-1+k)n+5(-4+k)n^2}}\right)}\right)}{2-2k+n}}$

\noindent$\scriptsize{\textsf{In[14]:=}}$ \textbf{CylindricalDecomposition[Implies[}$\boldsymbol{n\geq 2k\&\&n\geq 19\&\&k_1<k<k_2,G[n,k]\geq  0],\{n, k\}]}$

\noindent$\scriptsize{\textsf{Out14]=}}$ True

This completes the proof. $\Box$
\end{proof}

$\mathbf{Remark \ 1}$. The reason why we do not reduce \eqref{2.14} further is to save computational time.

\begin{theorem}\label{thm2.4}
The sequence $\{JL_{n,k}\}_{k\geq0}$ is log-concavity.
\end{theorem}

\begin{proof}
By \eqref{2.1} and \eqref{2.2}, we obtain
\begin{equation*}
    \begin{split}
    &JL_{n,k}^2-JL_{n,k+1}JL_{n,k-1}\\&
    =JL_{n,k}^2-\frac{1}{5(k+1)(n+1)^2(n-2k+2)}(-n(n-2k+1)JL_{n+1,k}+\\&(3n-5k)(n+1)JL_{n,k})(knJL_{n+1,k}+2k(n+1)JL_{n,k})\\&
    =\frac{JL_{n,k}^2}{5(k+1)(n+1)^2(n-2k+2)}\\&
    \Big(kn^2(n-2k+1)\left(\frac{JL_{n+1,k}}{JL_{n,k}}\right)^2-k(n-k-2)n(1+n)\left(\frac{JL_{n+1,k}}{JL_{n,k}}\right)\\&
    -(1+n)^2(-10+(-5+k)n)\Big).
    \end{split}
\end{equation*}
Set
\begin{equation*}
    g_{n,k}(x):=kn^2(n-2k+1)x^2-k(n-k-2)n(1+n)x -(1+n)^2(-10+(-5+k)n).
\end{equation*}
Since $\frac{JL_{n,k}^2}{5(k+1)(n+1)^2(n-2k+2)}> 0$, it remains to show that $ g_{n,k}(x)\geq 0$. Denote the discriminant of $g_{n,k}(x)$ as
\begin{equation*}
    delt(n,k)=(k(n-k-2)n(1+n))^2+4kn^2(n-2k+1)(1+n)^2(-10+(-5+k)n).
\end{equation*}
Using CAD, we compute an equivalent condition for the $delt(n,k)>0$.

\noindent$\scriptsize{\textsf{In[15]:=}}$ $\boldsymbol{delt[n\_,k\_]:=(k(n-k-2)n(1+n))^2+4kn^2(n-2k+1)(1+n)^2(-10+(-5+k)n)}$

\noindent$\scriptsize{\textsf{In[16]:=}} $\textbf{CylindricalDecomposition[Implies[}$\boldsymbol{n\geq 2k\&\&n\geq 10\&\&}$$\boldsymbol{k> 0, delt[n,k]> 0],}$ $\boldsymbol{\{n, k\}]}$

\noindent$\scriptsize{\textsf{Out[16]=}}$ $n<10||(n\geq 10\&\&(k\leq 0||k> $Root$[-40-60n-20n^2+(84+40n+5n^2)\#1+(4-10n)\#1^2+\#1^3 \&, 1]))$

In other words, we obtain that the formula
\begin{equation*}
    n\geq 2k \wedge n\geq 10 \wedge k\geq 0 \Longrightarrow delt(n,k)> 0,
\end{equation*}
is equivalent over the reals to the formula
\begin{equation*}
    \begin{split}
    &n<10 \vee (n\geq 10 \wedge (k<0  \vee k> R_1)),
    \end{split}
\end{equation*}
where
\begin{equation*}
    R_1=Root(-40-60n-20n^2+(84+40n+5n^2)X+(4-10n)X^2+X^3 \&, 1).
\end{equation*}
The symbol Root $(-40-60n-20n^2+(84+40n+5n^2)X+(4-10n)X^2+X^3 \&, 1)$ denotes the algebraic function which maps $n$ to the smallest real root of the polynomial $-40-60n-20n^2+(84+40n+5n^2)X+(4-10n)X^2+X^3$.

Since the leading coefficient of $g_{n,k}(x)$ is positive, it is clearly that $g_{n,k}(x)\geq 0$ when $delt(n,k)\leq 0$. We just need to consider $n\geq 10,k>R_1$.

We compute an equivalent condition for the $g_{n,k}(x)\geq 0$,

\noindent$\scriptsize{\textsf{In[17]:=}}$ $\boldsymbol{R_1:=}$\textbf{Root}$\boldsymbol{[-40-60n-20n^2+(84+40n+5n^2)\#1+(4-10n)\#1^2+\#1^3 \&, 1]}$

\noindent$\scriptsize{\textsf{In[18]:=}}$ $\boldsymbol{g_{n,k}[x]:=kn^2(n-2k+1)x^2-k(n-k-2)n(1+n)x -(1+n)^2(-10+(-5+k)n)}$

\noindent$\scriptsize{\textsf{In[19]:=}}$ \textbf{CylindricalDecomposition[Implies[}$\boldsymbol{n\geq 2k\&\&n\geq 10\&\&k>R_1,g_{n,k}[x]\geq  0]}$ $\boldsymbol{,\{n, k, x\}]}$

\noindent$\scriptsize{\textsf{Out[19]=}}$ $n<10||(n\geq0 \&\&(k\leq$Root$[-40-60n-20n^2+(84+40n+5n^2)\#1+(4-10n)\#1^2+\#1^3 \&, 1])||($Root$[-40-60n-20n^2+(84+40n+5n^2)\#1+(4-10n)\#1^2+\#1^3 \&, 1]<k\leq \frac{n}{2}\&\&(x\leq -\frac{2+k+n+kn-n^2}{2n(1-2k+n)}-\frac{1}{2}\sqrt{\frac{H_1(n,k)}{k(-1+2k-n)^2n^2}}||x\geq  -\frac{2+k+n+kn-n^2}{2n(1-2k+n)}+\frac{1}{2}\sqrt{\frac{H_1(n,k)}{k(-1+2k-n)^2n^2}}))||k>\frac{n}{2}))$

Since $JL_{n,k}$ is sequence of positive integers, we only need to consider $x\geq -\frac{2+k+n+kn-n^2}{2n(1-2k+n)}+\frac{1}{2}\sqrt{\frac{H_1(n,k)}{k(-1+2k-n)^2n^2}}$. Meantime, we also find

\noindent$\scriptsize{\textsf{In[20]:=}}$ $\boldsymbol{H_1[n\_,k\_]:=-40+84k+4k^2+k^3-140n+208kn-2k^2n+2k^3n-180n^2}$

$\boldsymbol{+169kn^2-16k^2n^2+k^3n^2-100n^3+50kn^3-10k^2n^3-20n^4+5kn^4}$

\noindent$\scriptsize{\textsf{In[21]:=}}$ $\boldsymbol{l[n\_,k\_]:=-\frac{2+k+n+kn-n^2}{2n(1-2k+n)}+\frac{1}{2}\sqrt{\frac{H_1[n,k]}{k(-1+2k-n)^2n^2}}}$

\noindent$\scriptsize{\textsf{In[22]:=}}$ \textbf{CylindricalDecomposition[Implies[}$\boldsymbol{n\geq 2k\&\&n\geq 19\&\&k>R_1,}$

$\boldsymbol{\frac{(n-k)(n+1)}{n(n-2k+1)}-l[n,k] \geq  0],\{n, k\}]}$

\noindent$\scriptsize{\textsf{Out[22]:=}}$ $n<19||(n\geq 19\&\&(k\leq \frac{20+9n+n^2}{4(-1+n)}-\frac{1}{4}\sqrt{\frac{480+400n+41n^2-22n^3+n^4}{(-1+n)^2}}||k\geq \frac{20+9n+n^2}{4(-1+n)}+\frac{1}{4}\sqrt{\frac{480+400n+41n^2-22n^3+n^4}{(-1+n)^2}}))$

We know that $L(n,k)\geq l(n,k)$ for $n\geq 19, k\leq k_1 $ or $k\geq k_2$. By lemma \ref{lem2.2}, we have

\noindent$\scriptsize{\textsf{In[23]:=}}$ \textbf{CylindricalDecomposition[Implies[}$\boldsymbol{n\geq 2k\&\&n\geq 19\&\&R_1<k\leq k_1\&\&}$

$\boldsymbol{x\geq \frac{(n-k)(n+1)}{n(n-2k+1)}, g_{n,k}[x]\geq 0],\{n, k, x\}]}$

\noindent$\scriptsize{\textsf{Out[23]=}}$ True

\noindent$\scriptsize{\textsf{In[24]:=}}$ \textbf{CylindricalDecomposition[Implies[}$\boldsymbol{n\geq 2k\&\&n\geq 19\&\&k\geq k_2\&\&}$

$\boldsymbol{x\geq \frac{(n-k)(n+1)}{n(n-2k+1)},g_{n,k}[x]\geq 0],\{n, k, x\}]}$

\noindent$\scriptsize{\textsf{Out[24]=}}$ True

From which we gain that $g_{n,k}(x)\geq 0$ for $n\geq 19, R_1< k\leq k_1 $ or $k\geq k_2$. Following we prove that $g_{n,k}(x)\geq 0$ for $n\geq 19$, $k_1<k<k_2$. By lemma \ref{2.3}, we have

\noindent$\scriptsize{\textsf{In[25]:=}}$ \textbf{CylindricalDecomposition[Implies[}$\boldsymbol{n\geq 2k\&\&n\geq 19\&\&k_1<k<k_2\&\&}$

$\boldsymbol{x\geq l[n,k]),g_{n,k}[x])\geq 0],\{n, k, x\}]}$

\noindent$\scriptsize{\textsf{Out[25]=}}$ True

Thus we get the proof for $n\geq 19$. For $1\leq n\leq 18$, we just need a simple computation for Table 1. $\Box$
\end{proof}

By similar proof, we also show that the other two sequences A037027 and A073370 \cite{NJ} in OEIS are log-concavity.

\section{The mode of the sequence}
\begin{theorem}\label{thm3.1}
The sequence $\{JL_{n,k}\}_{k\geq 0}$ is strictly increasing in the left of  mode $\lfloor \frac{n-4}{6} \rfloor+1$ for $n\geq 4$.
\end{theorem}
Denote
\begin{equation*}
    \bigtriangleup(n,k)=JL_{n,k+1}-JL_{n,k}
\end{equation*}
then the theorem equals to
\begin{equation*}
    \begin{cases}
        \bigtriangleup(n,k)\leq 0,~~~~if~~2k\leq n\leq 6k+3,\\
        \bigtriangleup(n,k)\geq 0,~~~~if~~n\geq 6k+4.
    \end{cases}
\end{equation*}

\begin{lemma}\label{lem3.2}For any positive integers n and k with $n\geq 2k$, we have
\begin{equation}\label{3.1}
    \sum_{j=0}^{2}d_{j}(n,k)\bigtriangleup(n+j,k)=0,
\end{equation}
where
\begin{equation*}
    \begin{split}
    &d_0(n,k)=(1+n)(2+n)(-2-6k+n)(-3k+n),\\&
    d_1(n,k)=(2+n)(3+12k+3k^2-18k^3+8kn+27k^2n-2n^2-10kn^2+n^3),\\&
    d_2(n,k)=-(1+n)(-3-6k+n)(-1-3k+n)(2-2k+n).
    \end{split}
\end{equation*}
\end{lemma}

\begin{proof}
By Eq. \eqref{1.1},
\begin{equation}\label{3.2}
    \begin{split}
    &\bigtriangleup(n,k)=\sum_{i=k}^{\lfloor n/2 \rfloor}\frac{n}{n-i}\binom{n-i}{i}\binom{i}{k}\frac{i-2k-1}{k+1}.
    \end{split}
\end{equation}
It is similar to the proof for Lemma \ref{2.1}, by using the RISC package fastZeil.m on $n$, we have

\noindent$\scriptsize{\textsf{In[26]:=}}$ \textbf{Zb[}$\boldsymbol{\frac{n}{n-i}}$\textbf{Binomial}$\boldsymbol{[n-i,i]}$\textbf{Binomial}$\boldsymbol{[i, k]\frac{i-2k-1}{k+1},\{i, k,}$   \textbf{Infinity\},}$\boldsymbol{n, 2]}$

\noindent$\scriptsize{\textsf{Out[26]=}}$ \{$(1+n)(2+n)(-2-6k+n)(-3k+n)$SUM$[n]+(2+n)(3+12k+3k^2-18k^3+8kn+27k^2n-2n^2-10kn^2+n^3)$SUM$[1+n]-(1+n)(-3-6k+n)(-1-3k+n)(2-2k+n)$SUM$[2+n]==0$\}

This completes the proof. $\Box$
\end{proof}

\begin{proposition}\label{pro3.3} For any nonnegative integers n and k, then $\bigtriangleup(n,k)>0$ for $n\geq 6k+4$.
\end{proposition}
\begin{lemma}\label{lem3.4} For any nonnegative integers n and k with $n\geq 6k+4$,  $\bigtriangleup(n,k)$ is increasing, i.e.,
\begin{equation}
   \Phi(n,k)=\bigtriangleup(n+1,k)-\bigtriangleup(n,k)>0.
\end{equation}\label{3.3}
\end{lemma}
\pf With an induction proof on $n$  which equals the following  formula,
\begin{equation*}
    \forall n\in \mathbb{N}~~ \forall k\in \mathbb{N}:(n\geq 6k+4\wedge\Phi(n,k))\Longrightarrow \Phi(n+1,k).
\end{equation*}
Applying the three recurrence relations
 $\Phi(n+1,k)$ is found to be equivalent to
\begin{equation}\label{3.4}
    \begin{split}
    & \bigtriangleup(n+2,k)> \bigtriangleup(n+1,k)\\&
    \Longleftrightarrow-\frac{d_0(n,k)}{d_2(n,k)}\bigtriangleup(n,k)-\frac{d_1(n,k)}{d_2(n,k)}\bigtriangleup(n+1,k)>\bigtriangleup(n+1,k).
    \end{split}
\end{equation}
Based on the application of CAD, where the  real variables $D_0,D_1$ represent $\bigtriangleup(n,k),\bigtriangleup(n+1,k)$, then the proof equals

\noindent$\scriptsize{\textsf{In[27]:=}}$ $\boldsymbol{d_0[n\_,k\_]:=(1+n)(2+n)(-2-6k+n)(-3k+n)}$

\noindent$\scriptsize{\textsf{In[28]:=}}$ $\boldsymbol{d_1[n\_,k\_]:=(2+n)(3+12k+3k^2-18k^3+8kn+27k^2n-2n^2-10kn^2+n^3)}$

\noindent$\scriptsize{\textsf{In[29]:=}}$ $\boldsymbol{d_2[n\_,k\_]:=-(1+n)(-3-6k+n)(-1-3k+n)(2-2k+n)}$

\noindent$\scriptsize{\textsf{In[30]:=}}$ \textbf{CylindricalDecomposition[Implies[}$\boldsymbol{n\geq 6k+4\&\&D_0>0\&\&D_1>0\&\&k\geq 0}$

$\boldsymbol{\&\&D_1>D_0,-\frac{d_0[n,k]}{d_2[n,k]}D_0-\frac{d_1[n,k]}{d_2[n,k]}D_1>D_1],\{n, k, D_0, D_1\}]}$

\noindent$\scriptsize{\textsf{Out[30]=}}$ True

The proof is completed  for all $n\geq 6k+4$. $\Box$

\begin{lemma}\label{lem3.5} For any nonnegative integer $k$, $\bigtriangleup(6k+4,k)$ is positive.
\end{lemma}
\begin{proof}
By \eqref{3.2}, we obtain
\begin{equation}
    \bigtriangleup(6k+4,k)=\sum_{i=k}^{3k+2}\frac{6k+4}{6k+4-i}\binom{6k+4-i}{i}\binom{i}{k}\frac{i-2k-1}{k+1}.
\end{equation}
By using the RISC package fastZeil.m on $k$, we have

\noindent$\scriptsize{\textsf{In[31]:=}}$ \textbf{Zb[}$\boldsymbol{\frac{6k+4}{6k+4-i}}$\textbf{Binomial}$\boldsymbol{[6k+4-i,i]}$\textbf{Binomial}$\boldsymbol{[i, k]\frac{i-2k-1}{k+1},\{i, k,3k+2\}, k, 2]}$

If `2+2k' is a natural number and `2+3k' is no negative integer, then:

\noindent$\scriptsize{\textsf{Out[31]=}}$ \{$72(3+2k)(4+3k)(5+3k)(7+3k)(8+3k)(5+6k)(7+6k)(2640+2681k+671k^2)$SUM$[k]-(2+k)(7+3k)(8+3k)(462369600+2067513120k+3748025842k^2+3537926637k^3+1838345599k^4+499358523k^5+55457479k^6)$ SUM$[1+k]+40(2+k)(3+k)(5+2k)(4+3k)(5+3k)(9+4k)(11+4k)(630+1339k+671k^2)$SUM$[2+k]==0$\}

In other words, we have
\begin{equation*}
    \sum_{j=0}^{2}b_j(k)A(k+j)=0,
\end{equation*}
where
\begin{equation*}
    \begin{split}
    & A(k)=\bigtriangleup(6k+4,k),\\&
    b_0(k)=72(3+2k)(4+3k)(5+3k)(7+3k)(8+3k)(5+6k)(7+6k)(2640\\&+2681k+671k^2),\\&
    b_1(k)=-(2+k)(7+3k)(8+3k)(462369600+2067513120k\\&+3748025842k^2+3537926637k^3+1838345599k^4+499358523k^5\\&+55457479k^6),\\&
    b_2(k)=40(2+k)(3+k)(5+2k)(4+3k)(5+3k)(9+4k)(11+4k)(630\\&+1339k+671k^2).
    \end{split}
\end{equation*}
Using, real variables $D_0,D_1,D_2$ for representing $A(k),A(k+1),A(k+2)$, we can now prove the induction step formula by another application of CAD:

\noindent$\scriptsize{\textsf{In[32]:=}}$ $\boldsymbol{b_0[k\_]:=72(3+2k)(4+3k)(5+3k)(7+3k)(8+3k)(5+6k)(7+6k)}$

$\boldsymbol{(2640+2681k+671k^2)}$

\noindent$\scriptsize{\textsf{In[33]:=}}$ $\boldsymbol{b_1[k\_]:=-(2+k)(7+3k)(8+3k)(462369600+2067513120k+3748025842k^2}$

$\boldsymbol{+3537926637k^3+1838345599k^4+499358523k^5+55457479k^6)}$

\noindent$\scriptsize{\textsf{In[34]:=}}$ $\boldsymbol{b_2[k\_]:=40(2+k)(3+k)(5+2k)(4+3k)(5+3k)(9+4k)(11+4k)(630+1339k+671k^2)}$

\noindent$\scriptsize{\textsf{In[35]:=}}$ \textbf{CylindricalDecomposition[Implies[}$\boldsymbol{D_0>0\&\&D_1>0\&\&k\geq 0\&\&D_1>D_0}$

$\boldsymbol{\&\&b_0[k]D_0+b_1[k]D_1+b_2[k]D_2==0,D_2>D_1], \{k, D_0, D_1, D_2\}]}$

\noindent$\scriptsize{\textsf{Out[35]=}}$ True

Since $A(k)$ is strictly increasing for $k\geq0$, with initial value $A(0)=1$, thus we get the proof. $\Box$
\end{proof}

\begin{proposition}\label{pro3.6}
For any nonnegative integers $k$ when $n\geq 4$,  $\bigtriangleup(n,k)\leq 0$ for $2k\leq n\leq 6k+3$.
\end{proposition}

\begin{theorem}\label{lem3.7}
 For any nonnegative integers $k$, $\bigtriangleup(6k+3,k)$ is negative.
\end{theorem}
\pf By \eqref{3.2}, we obtain
\begin{equation}
    \bigtriangleup(6k+3,k)=\sum_{i=k}^{3k+1}\frac{6k+3}{6k+3-i}\binom{6k+3-i}{i}\binom{i}{k}\frac{i-2k-1}{k+1}.
\end{equation}
It is similar to the proof for lemma \ref{lem3.5}, we have

\noindent$\scriptsize{\textsf{In[36]:=}}$ \textbf{Zb[}$\boldsymbol{\frac{6k+3}{6k+3-i}}$\textbf{Binomial}$\boldsymbol{[6k+3-i,i]}$\textbf{Binomial}$\boldsymbol{[i, k]\frac{i-2k-1}{k+1},\{i, k,3k+1\}, k, 2]}$

If `1+2k' is a natural number and `k' is no negative integer, then:

\noindent$\scriptsize{\textsf{Out[36]=}}$ \{$72(3+2k)^2(2+3k)(4+3k)(7+3k)(5+6k)(7+6k)(2640+2681k+671k^2)$SUM$[k]-(2+k)(1+2k)(7+3k)(462369600+2067513120k+3748025842k^2+3537926637k^3+1838345599k^4+499358523k^5+55457479k^6)$SUM$[1+k]+40(2+k)(3+k)(1+2k)(3+2k)(4+3k)(9+4k)(11+4k)(630+1339k+671k^2)$SUM$[2+k]==0$\}

In other words, we have
\begin{equation*}
    \sum_{j=0}^{2}c_{j}(k)B(k+j),
\end{equation*}
where
\begin{equation*}
    \begin{split}
    & B(k)=\bigtriangleup(6k+3,k),\\&
    c_0(k)=72(3+2k)^2(2+3k)(4+3k)(7+3k)(5+6k)(7+6k)(2640+2681k\\&+671k^2),\\&
    c_1(k)=-(2+k)(1+2k)(7+3k)(462369600+2067513120k+\\&3748025842k^2+3537926637k^3+1838345599k^4+499358523k^5\\&+55457479k^6),\\&
    c_2(k)=40(2+k)(3+k)(1+2k)(3+2k)(4+3k)(9+4k)(11+4k)(630\\&+1339k+671k^2).
    \end{split}
\end{equation*}
Using, real variables $D_0,D_1,D_2$ for representing $B(k),B(k+1),B(k+2)$, we can now prove the induction step formula by another application of CAD:

\noindent$\scriptsize{\textsf{In[37]:=}}$ $\boldsymbol{c_0[k\_]:=72(3+2k)^2(2+3k)(4+3k)(7+3k)(5+6k)(7+6k)(2640+2681k+671k^2)}$

\noindent$\scriptsize{\textsf{In[38]:=}}$ $\boldsymbol{c_1[k\_]:=-(2+k)(1+2k)(7+3k)(462369600+2067513120k+3748025842k^2}$

$\boldsymbol{+3537926637k^3+1838345599k^4+499358523k^5+55457479k^6)}$

\noindent$\scriptsize{\textsf{In[39]:=}}$ $\boldsymbol{c_2[k\_]:=40(2+k)(3+k)(1+2k)(3+2k)(4+3k)(9+4k)(11+4k)(630+1339k+671k^2)}$

\noindent$\scriptsize{\textsf{In[40]:=}}$ \textbf{CylindricalDecomposition[Implies[}$\boldsymbol{D_0<0\&\&D_1<0\&\&k\geq 0\&\&D_1<D_0}$

$\boldsymbol{\&\&c_0[k]D_0+c_1[k]D_1+c_2[k]D_2==0,D_2<D_1], \{k, D_0, D_1, D_2\}]}$

\noindent$\scriptsize{\textsf{Out[40]=}}$ True

Then $B(k)$ is strictly decreasing for $k\geq 0$, since the initial value $B(0)=-1$, thus we get the proof.
%

Combine proposition \ref{pro3.3}, proposition \ref{pro3.6} and theorem \ref{lem3.7}, it is clearly that our main theorem is true.

$\mathbf{Remark \ 2}$. All RISC software packages mentioned in this paper are available online at

\centerline{https://caa.risc.jku.at/software}

\end{CJK*}

\begin{thebibliography}{99}

\bibitem{FB1989} F. Brenti, Unimodal, log-concave and P\'{o}lya frequency sequences in combinatorics. Mem. Amer. Math. Soc. (1989)81(413):viii+106.

\bibitem{FB1994} F. Brenti, Log-concave and unimodal sequences in algebra, combinatorics, and geometry: an update, Jerusalem Combinatorics' 93: An International Conference in Combinatorics, May 9-17, 1993, Jerusalem, Israel. Vol. 178.  Amer. Math. Soc. (1994)71每89.

\bibitem{EC1975} G. E. Collins. Quantifier elimination for real closed fields by cylindrical algebraic decomposition. In Automata theory and formal languages (Second GI Conf., Kaiserslautern), volume Vol. 33 of Lecture Notes in Comput. Sci., pages 134每183. Springer, Berlin-New York, 1975.
\bibitem{hou2019} Qing-Hu Hou, Zuo-Ru Zhang, Asymptotic r-log-convexity and P-recursive sequences. J. Symb. Comput. (2019)93: 21-33.
\bibitem{hou2021} Qing-Hu Hou, G. Li, Log-concavity of P-recursive sequences[J]. J. Symb. Comput. (2021) 107: 251-268.
\bibitem{KP2007}  M. Kauers, P. Paule, A computer proof of Moll's log-concavity conjecture. Proc. Amer. Math. Soc., 135(12):3847每3856, 2007.
\bibitem{Mao2023} J. Mao, Y. Pei, The asymptotic log-convexity of Ap$\acute{e}$ry-like numbers[J]. J. Differ. Equations Appli., 2023, 29(8): 799-813.
\bibitem{PM1992} P. Paule,  M. Schorn, A Mathematica version of Zeilberger's algorithm for proving binomial coefficient identities[J]. J. Symb. Comput.(1995) 20(5-6): 673-698.

\bibitem{AR2001} A. Riese,  Fine-tuning Zeilberger＊s algorithm: the methods of automatic filtering and creative substituting. In Garvan, F. G. et al. eds, Symbolic Computation, Number Theory, Special Functions, Physics and Combinatorics, Developments in Mathematics 4, (2001)243每254.
\bibitem{NJ} N.J.A. Sloane, The on-line encyclopedia of integer sequences.

\bibitem{PS1989} R.P. Stanley, Log-concave and unimodal sequences in algebra, combinatorics, and geometry. Ann. New York Acad. Sci 576.1 (1989): 500-535.

\bibitem{Sun1975} Y. Sun, Numerical triangles and several classical sequences. Fibonacci Quart.,43(4):359每370, 2005.

\bibitem{SW1992} H. S. Wilf, D. Zeilberger. An algorithmic proof theory for hypergeometric (ordinary and q) multisum/integral identities. Invent. Math., 108:575每633, 1992.
\bibitem{Zhu2020} B.-X. Zhu,  A generalized Eulerian triangle from staircase tableaux and tree-like tableaux, J. Combin. Theory Ser. A 172 (2020) 105206.
\end{thebibliography}
\end{document}